\documentclass{article}
\usepackage{amssymb}
\usepackage{tikz-cd}
\usepackage{amscd}
  \usepackage{graphicx}
 \usepackage{amsmath,amsthm,amssymb,amsfonts}
  \usepackage{indentfirst}
\newtheorem{theorem}{Theorem}
\newtheorem{example}{Example}

\newtheorem{corollary}{Corollary}
\newtheorem{conjecture}{Conjecture}

\newtheorem{lemma}{Lemma}

\newcommand{\Z}{{\mathbb Z}}
\newcommand{\N}{{\mathbb N}}

\title{Integers representable as differences of linear recurrence sequences}
\author{Daodao Yang}
\begin{document}
\date{}
\maketitle

\centerline{\bf Abstract}
\medskip
\maketitle

    Let $\{U_n\}_{n \geqslant 0}$ and $\{G_m\}_{m \geqslant 0}$ be two linear recurrence sequences defined over the integers.  We establish  an asymptotic formula  for the number of integers $c$ in the range $[-x, x]$ which can be represented as differences $ U_n - G_m$, when $x$ goes to infinity. In particular, the density of such integers is $0$.
\bigskip

\section{Introduction and main results}
  S. S. Pillai \cite{PillaiConjecture} \cite{Pillaicorrection} studied the following Diophantine equation.  

\begin{equation}\label{Pequation}
    a^m - b^n = c
\end{equation}

    He conjectured that for arbitrary given integer $c \geqslant 1$ the Diophantine equation $(\ref{Pequation})$ has only finitely many positive integer solutions $(a, b, m, n)$ , with $m,n \geqslant 2.$ Pillai's conjecture is a corollary of the \emph{abc} conjecture. For $c = 1$, Pillai's conjecture coincides with  Catalan's conjecture which is already proved by Mihăilescu \cite{Mihailescu}. For all $c \neq 1$, Pillai's conjecture is still open.

For fixed integers $a, b$, Pillai proved that for all sufficiently large $c$, there is at most one solution $(m, n)$ with $m, n \geqslant 2$ to equation $(\ref{Pequation})$.  Pillai  also \cite{Pillai}\cite{PillaiConjecture} \cite{Pillaicorrection} proved that that the number of integers $c$ in the range $[1, x]$ which can be expressed in the form $c = a^m - b^n$ is asymptotically equal to

\begin{equation*}
    \frac{(\log x)^2}{2 (\log a) (\log b)}
\end{equation*}

Recent years, there appeared several papers studying  solutions $(m, n)$ to the following Diophantine equation $(\ref{ReCurequation})$, where $\{U_n\}_{n \geqslant 0}$ and $\{G_m\}_{m \geqslant 0}$ are  given linear recurrence sequences:

\begin{equation}\label{ReCurequation}
    U_n - G_m = c 
\end{equation}

For instance, in \cite{FibonacciPower2}, the authors consider the special situation, $U_n = F_n$, the Fibonacci numbers, $G_m = 2^m$. In \cite{TribonacciPower2}, $U_n = T_n$, the Tribonacci numbers, $ G_m = 2^m$. In \cite{PinkI},  $U_n = F_n$, $G_m = T_m$. These papers find all integers $c$ having two different representations as $c = U_n - G_m$ for some 
integers $m \geqslant M_0, n \geqslant N_0$, where $M_0$ and $N_0$ are fixed integers depending on the recurrence sequences.

Furthermore, in \cite{Pink}, the authors consider the general linear recurrence sequences (with some subtle requirements). They prove that there exists a  effectively computable finite set $\mathcal{C}$ such that there exist at least two solutions  $(m, n)$ for the  equation $(\ref{ReCurequation})$ if and only if $c \in \mathcal{C}$.

In this paper, we consider the problem that how many $c$ can make equation (\ref{ReCurequation}) have at least one solution, where $|c| \leqslant x$. We find an find  an asymptotic formula for the number of such integers $c$ when $x \to \infty$. To state our result, we make some definitions first.
\bigskip

We say that  $\{U_n\}_{n \geqslant 0}$ is a linear recurrence sequence defined over the integers if for some positive integer $k$, we have
\begin{align*}
    U_{n+k} = \sum_{i=0}^{k-1}\xi_i U_{n+i},  ~~\forall n \geqslant 0,
\end{align*}

where $\xi_i \in \Z $ are fixed and $U_0, U_1, \dots, U_{k-1}$ are given integers.
\bigskip

The characteristic polynomial of $\{U_n\}_{n \geqslant 0}$ is defined as:

\begin{align*}
    f(X) = X^k - \sum_{i=0}^{k-1}\xi_i x^i = \prod_{i=1}^t(X - \alpha_i)^{\sigma_i},
\end{align*}

where $\sigma_i \in \N$  and $\alpha_1,  \dots, \alpha_{t}$ are distinct roots of $f(X)$, called characteristic roots of  $\{U_n\}_{n \geqslant 0}$. 
\bigskip

 Let  $\alpha_1, \alpha_2, \cdots, \alpha_t ~(t \geqslant 1)~$ be the characteristic roots of $\{U_n\}_{n \geqslant 0}$ and $\beta_1, \beta_2, \cdots, \beta_s ~(s \geqslant 1)~$ the characteristic roots of $\{G_m\}_{m  \geqslant 0}$. Suppose $|\alpha_1 | > |\alpha_2| \geqslant \dots \geqslant |\alpha_t|$ and $|\beta_1 | > |\beta_2| \geqslant \dots \geqslant |\beta_s|$.
 Then we call $\alpha = \alpha_1$ the \textbf{dominant root} of $\{U_n\}_{n \geqslant 0}$ and $\beta = \beta_1$ the dominant root of $\{G_m\}_{m  \geqslant 0}$. Furthermore, we call $\{U_n\}_{n \geqslant 0}$ and $\{G_m\}_{m \geqslant 0}$ are \textbf{dominant to each other} if we additionally have $|\alpha| > |\beta_2|$ (or $\beta$ is the unique characteristic root of $\{G_m\}_{m\geqslant 0}$ ) and $|\beta| > |\alpha_2|  $ (or $\alpha$ is the unique characteristic root of $\{U_n\}_{n \geqslant 0}$ ).

For a fixed integer $N_0$, a sequence $\{U_n\}_{n \geqslant 0}$ is said to be \textbf{strictly increasing in absolute values} for $n \geqslant N_0$ if $|U_{n+1}| > |U_{n}|> 0 $ for all $n \geqslant N_0$. This condition can ensure that  all elements in the sequence are different from each other, starting from some element.

\begin{theorem}\label{Main}
Let $\{U_n\}_{n \geqslant 0}$ and $\{G_m\}_{m \geqslant 0}$ be two linear recurrence sequences defined over the integers. Suppose that $\{U_n\}_{n \geqslant 0}$ and $\{G_m\}_{m \geqslant 0}$ are dominant to each other with dominant roots $\alpha$ and $\beta$ respectively. Furthermore, suppose that $\alpha$ and $\beta$ are multiplicatively independent, $|\alpha| >1 $ and  $|\beta| >1$. Suppose also that $\{U_n\}_{n \geqslant 0}$ and $\{V_m\}_{m \geqslant 0}$ are strictly increasing in absolute values for $n \geqslant N_0$ and $m \geqslant  M_0$ respectively. Then the number of integers $c$ in the range $[-x, x]$ which can be written in the form $c = U_n - G_m$ is asymptotically equal to $\frac{(\log x)^2}{\log |\alpha| \cdot \log |\beta|}.$   In other words,
 \medskip
 
$$  \# \{c:  |c| \leqslant x,~ c = U_n - G_m \}  \sim \frac{(\log x)^2}{\log |\alpha|\cdot  \log |\beta|},  ~~~as ~x \to \infty$$
\end{theorem}

\bigskip

\begin{corollary}
Assume the same conditions for $\{U_n\}_{n \geqslant 0}$ and $\{G_m\}_{m \geqslant 0}$ as in Theorem \ref{Main}. Then the density of integers of the form $U_n - G_m$ is $0$.
\end{corollary}

\bigskip

\begin{example}
The Fibonacci numbers $\{F_n\}_{n \geqslant 0}$ are defined by $F_{n+2} = F_n + F_{n+1}$ and $F_0 = 0, F_1 = 1 $. $\{F_n\}_{n \geqslant 0}$ has a dominant root $\alpha = \frac{\sqrt{5}+1}{2} > 1$.  Another characteristic root is $\frac{1-\sqrt{5}}{2}$ with absolute value less than 1.The  number of integers $c$ in the range $[-x, x]$ which can be written in the form $c = F_n - 2^m$ is asymptotically equal to

$$  \frac{(\log x)^2}{\log (\frac{\sqrt{5}+1}{2}) \cdot \log  2}$$

\end{example}

\bigskip

\begin{example}
The Tribonacci numbers $\{T_m\}_{m \geqslant 0}$ are defined by $T_{m+3} = T_m + T_{m+1} + T_{m+2} $ and $T_0 = 0, T_1 = 1, T_2 = 1 $. $\{T_m\}_{m \geqslant 0}$ has a dominant root $\alpha = \frac{1}{3}(1+  \sqrt[3]{19+3 \sqrt{33}}   + \sqrt[3]{19 -3 \sqrt{33}}  ) > 1$.  Two other characteristic roots of $\{T_m\}_{m \geqslant 0}$ are complex roots with absolute values less than 1. The  number of integers $c$ in the range $[-x, x]$ which can be written in the form $c = F_n - T_m$ is asymptotically equal to

$$  \frac{(\log x)^2}{\log (\frac{\sqrt{5}+1}{2}) \cdot \log  (\frac{1}{3}(1+ \sqrt[3]{19+3 \sqrt{33}}   + \sqrt[3]{19 -3 \sqrt{33}}  ) )}$$

\end{example}

\bigskip

By the following theorem in $\cite{Pink}$, we know that if   $|c|$ is sufficiently large, then there exists at most one solution $(n, m)$ such that $c = U_n - G_m$. 
\bigskip

\begin{theorem} [Chim, Pink, Ziegler]
Suppose that $\{U_n\}_{n \geqslant 0}$ and $\{V_m\}_{m \geqslant 0}$ be two linear recurrence sequences defined over the integers with dominant roots $\alpha$ and $\beta$ respectively. Furthermore, suppose that $\alpha$ and $\beta$ are multiplicatively independent. Suppose also that $\{U_n\}_{n \geqslant 0}$ and $\{V_m\}_{m \geqslant 0}$ are strictly increasing in absolute values for $n \geqslant N_0$ and $m \geqslant  M_0$ respectively. Then there exists a finite set $\mathcal{C}$ such that the integer $c$ has at least two distinct representations of the form $U_n - V_m$ with $n \geqslant N_0$ and $m \geqslant  M_0$, if and only if $c \in \mathcal{C}$. The set $\mathcal{C}$ is effectively computable.
\end{theorem}

\bigskip

So \textbf{Theorem \ref{Main}} is equivalent to a corollary of \textbf{Theorem \ref{THMinequality}}.
\bigskip

\begin{theorem}\label{THMinequality} Assume the same conditions for $\{U_n\}_{n \geqslant 0}$ and $\{G_m\}_{m \geqslant 0}$ as in Theorem \ref{Main}.
The number of solutions $(n, m)$ of the inequality $| U_n - G_m | \leqslant x$ is denoted by $S(x)$, then

\begin{align*}
 \frac{S(x)}{(\log x )^2} \geqslant \frac{1}{\log |\alpha|\cdot \log |\beta|} + \resizebox{0.34cm}{!}{$\Theta$}  ( \frac{\log \log x}{\log x}), ~\forall\,x \geqslant 3
\end{align*}

Moreover, $\forall\, \eta \in (0, 1),$ $\exists \,x_1 = x_1 (\eta)$ depending on $\eta$ such that 

\begin{align*}
 \frac{S(x)}{(\log x )^2} \leqslant (1+ \eta)^2\frac{1}{\log |\alpha|\cdot \log |\beta|} + \resizebox{0.34cm}{!}{$\Theta$}( \frac{\log \log x}{\log x}), ~\forall\, x \geqslant x_1
\end{align*}

\end{theorem}
\medskip

\textbf{Remark:}~we use the theta notation $\Theta(\cdot)$, where $g = \Theta(h)$ means that $ c_1 h \leqslant  |g| \leqslant C_2 h $ for some positive constants $c_1$ and $C_2$. Moreover, if both the dominant roots $\alpha$ and $\beta$ have multiplicity equal to one, then the remainder terms $\Theta( \frac{\log \log x}{\log x})$ can be improved to  $\Theta( \frac{1}{\log x})$.
\bigskip

\begin{corollary}
Assume the same conditions for $\{U_n\}_{n \geqslant 0}$ and $\{G_m\}_{m \geqslant 0}$ as in Theorem \ref{Main}.
The number of solutions $(n, m)$ of the inequality $| U_n - G_m | \leqslant x$ is denoted by $S(x)$, then

\begin{align*}
 \lim_{x \to \infty} \frac{S(x)}{(\log x )^2} = \frac{1}{\log |\alpha|\cdot \log |\beta|}
\end{align*}

\end{corollary}

\section{Preliminary results}

Estimates for lower bounds for linear forms in logarithms are key tools for our proof.  Baker and Wüstholz have obtained many results. For instance, \cite{Baker} provided an explicit bound. And in this paper, we will use the theorem proved by  Matveev \cite{Matveev}. 

For an algebraic number $\alpha$, suppose its minimal primitive polynomial over the integers is 

$$P(z) = a (z - z_1) (z - z_2) \cdots (z - z_d),$$

then the absolute logarithmic height  of $\alpha $  is equal to:

$$ h(\alpha):~ = \frac{1}{d}  (\log a + \sum_{ j = 1}^{d} \log  (\emph{\emph{\textbf{max}}} \{1,  |z_j|\}))$$

In the following is a modified version \cite{Almostpowers} of Matveev’s theorem \cite{Matveev}.

\begin{theorem} [Matveev]
Let $\gamma_1, \dots, \gamma_t$ be non-zero elements  in a number field $\mathbb{K}$ of degree $D$, and let $b_1, \dots, b_t$ be rational integers.

$$B = \emph{\textbf{max}} \{ |b_1|, \dots ,  |b_t|\}$$

and 

$$ A_i \geqslant \emph{\textbf{max}} \{ D h(\gamma_i), |\log \gamma_i|, 0.16 \}, ~~1 \leqslant i \leqslant t.$$

Assume that 

$$\Lambda : = \gamma_1^{b_1} \cdots \gamma_t^{b_t}  - 1, $$

is non-zero. Then,

$$\log |\Lambda| > -3 \times 30^{t+4} \times (t + 1)^{5.5} \times D^2 (1 + \log D) (1 + \log tB) A_1 \cdots A_t. $$
\end{theorem}
\bigskip

In order to use Matveev's theorem, one needs to require that $ \Lambda \neq 0 $.  The following two lemmas can show that $ \Lambda \neq 0 $ if $n$ and $m$ are sufficiently large.

\bigskip

\begin{lemma}[\cite{Pink}]\label{Fraclm}

Let $\mathbb{K}$ be a number field and suppose that $\alpha, \beta \in \mathbb{K}$ are algebraic numbers which are multiplicaticely independent. Then there exists an effetively computable constant $C_0 > 0$ such that
\bigskip

$$ h(\frac{\alpha^n}{\beta^m})\geqslant C_0\, \emph{\textbf{max}}\{|n|, |m|  \}, ~~~\forall n,m \in \Z $$
\end{lemma}
\bigskip

\begin{lemma}[\cite{Pink}]\label{Loglm}
Let $\mathbb{K}$ be a number field. $p, q \in \mathbb{K}[x]$ are two arbitrary polynomials. Then there exists an effective constant $C$ depends on $p$ and $q$ such that
\bigskip

$$ h(\frac{p(n)}{q(m)})\leqslant C \log \emph{\textbf{max}}\{n, m  \}, ~~~ \forall n,m \in \Z, ~~  n,m \geqslant 2$$
\end{lemma}

\section{ Proof of Theorem  \ref{THMinequality}}

By our assumptions, we can write $U_n$ and $G_m$ as the following, where $a_i(n)$ and $b_j(m)$ are some polynomials.

\begin{align*}
    U_n  = \sum_{i=1}^{t}a_i(n)\alpha_i^n,~~~~~ 
    G_m  = \sum_{j=1}^{s}b_j(m)\beta_j^m
\end{align*}

In this section, when we mean a constant, it may depend on $ \alpha_i, \beta_j$, the polynomials $a_i$, and $b_j$. But we may  not point this out.

\bigskip

\begin{proof}

Using the property of dominant roots,  we can fix positive integers $L$ and $d$, such that

\begin{equation}
    |U_n| \leqslant L \cdot n^{d} \cdot |\alpha|^n,~~ \forall \, n \in \N,
    \end{equation}

\begin{equation}
    |G_m| \leqslant L \cdot m^{d} \cdot |\beta|^m,~~ \forall \, m \in \N,   
\end{equation}

We will need the following lemma to get estimates for upper bounds for $n$ and $m$.

\begin{lemma}
$\exists \, \epsilon_0 \in (0, 1)$,  $\forall \epsilon \in (0, \epsilon_0],$ we can find a constant $K_\epsilon  > 1$, such that 
 
\begin{equation}\label{KforU}
    |U_n| \leqslant K_\epsilon  \cdot (|\alpha| + \epsilon)^n,~~ \forall \, n \in \N,
    \end{equation}

\begin{equation}\label{KforG}
    |G_m| \leqslant K_\epsilon  \cdot (|\beta| + \epsilon)^m,~~ \forall \, m \in \N,   
\end{equation}

where the constant $K_\epsilon$ can be given by $$K_\epsilon = \frac{\kappa}{\epsilon^d},~~and~~ \kappa = 2\,\textbf{max}\{|\alpha|^d, |\beta|^d \}\cdot L  d^d  e^{-d}\,.$$
\end{lemma}

\begin{proof}

Without loss of generality, let's assume that $|\alpha| \geqslant |\beta|.$ Then (\ref{KforU}) and (\ref{KforG})  will follow from the following inequality

$$ L \cdot n^d \cdot |\alpha|^n \leqslant K_{\epsilon} \cdot (|\alpha| + \epsilon)^n  $$

Take logarithm and we can obtain

\begin{equation}\label{LogLK}
    \log L - \log K_{\epsilon}  \leqslant n \log (1 + \frac{\epsilon}{|\alpha|}) - d \log n 
    \end{equation}

Define $$ f(x): \,= x \log (1 + \frac{\epsilon}{|\alpha|}) - d \log x,~ \forall x \in (0, \infty).$$ Then $f(x)$  attains its minimum at $x_0 = \frac{d}{\log (1+\frac{\epsilon}{|\alpha|})}.$

So in order to make inequality (\ref{LogLK})  hold for all positive integer $n$, it suffices that the following hold.

$$  \log L - \log K_{\epsilon} \leqslant d - d \log ( \frac {d} {\log (1 + \frac{\epsilon}{|\alpha|}) } ) $$

Thus we  have

\begin{align*}
 \log K_{\epsilon}  \geqslant  - d \log \log (1 + \frac{\epsilon}{|\alpha|}) + \log L + d \log d - d     
\end{align*}

Therefore, 
\begin{align*}
 K_{\epsilon}  &\geqslant  \emph{\emph{exp}} (- d \log \log (1 + \frac{\epsilon}{|\alpha|})  ) \, \emph{\emph{exp}} (\log L + d \log d - d )\\
&= \frac{1}{(\log (1 + \frac{\epsilon}{|\alpha|}))^d} \cdot L d^d e^{-d}
\end{align*}

Note that $$  \lim_{x \to \infty} \frac{\log (1+x)}{x}  = 1\,, $$ so exists $\epsilon_0 \in (0, 1)$, when $\epsilon \in (0, \epsilon_0],$ we can take

$$ K_{\epsilon} = \frac{ \kappa }{\epsilon^d}~, $$

where $  \kappa = 2   |\alpha|^d  \cdot L  d^d  e^{-d}.$

\end{proof}

\subsection{A lower bound for $S(x)$}

The conditions  $K_\epsilon  \cdot (|\alpha| + \epsilon)^n \leqslant \frac{x}{2}$  and $K_\epsilon  \cdot (|\beta| + \epsilon)^m \leqslant \frac{x}{2}$ will imply

$$ |U_n - G_m| \leqslant |U_n| + |G_m|  \leqslant  K_\epsilon  \cdot (|\alpha| + \epsilon)^n + K_\epsilon  \cdot (|\beta| + \epsilon)^m \leqslant x.$$
\medskip

Noting that
\begin{align*}
    n \leqslant \frac{\log x}{\log (|\alpha| +\epsilon) } - \frac{\log(2K_\epsilon)}{\log (|\alpha| +\epsilon)} &\Longleftrightarrow K_\epsilon  \cdot (|\alpha| + \epsilon)^n \leqslant \frac{x}{2}\\
    \\
   m \leqslant \frac{\log x}{\log (|\beta| +\epsilon) } - \frac{\log(2K_\epsilon)}{\log (|\beta| +\epsilon)} &\Longleftrightarrow K_\epsilon  \cdot (|\beta| + \epsilon)^m \leqslant \frac{x}{2}
    \end{align*}
\bigskip

we can obtain lower bounds for $S(x)$:

\begin{align*}
    S(x) \geqslant (\frac{\log x}{\log (|\alpha| +\epsilon) } - \frac{\log(2K_\epsilon)}{\log (|\alpha| +\epsilon)})(\frac{\log x}{\log (|\beta| +\epsilon) } - \frac{\log(2K_\epsilon)}{\log (|\beta| +\epsilon)} )
\end{align*}
\medskip


Expand the product and divide $(\log x)^2:$

\begin{align*}
    \frac{S(x)}{  (\log x)^2} \geqslant &\frac{1}{\log( |\alpha| + \epsilon)  \log(|\beta| + \epsilon)}  -  \frac{ 2 \log (2 K_{\epsilon}) }{ \log( |\alpha| + \epsilon)  \log(|\beta| + \epsilon) } \frac{ 1} { \log x}\\\\  &+  \frac{ (\log(2 K_{\epsilon}))^2  }{ \log( |\alpha| + \epsilon)  \log(|\beta| + \epsilon) } \frac{1}{(\log x)^2}  \end{align*}

Set $$ \epsilon = \frac{ 1} { \log x} ~, $$ then 

$$   \log(2 K_{\epsilon}) = \log  \frac{ 2 \kappa }{\epsilon^d}  = \log (  2 \kappa (\log x)^d ) = \log 2 \kappa + d \log \log x =  \Theta ( \log \log x )$$

Thus

\begin{equation*}
 -  \frac{ 2 \log (2 K_{\epsilon}) }{ \log( |\alpha| + \epsilon)  \log(|\beta| + \epsilon) } \frac{ 1} { \log x}  =   \Theta (  \frac{ \log \log x }{   \log x }  )     
\end{equation*}

\begin{align*}
    \frac{ (\log(2 K_{\epsilon}))^2  }{ \log( |\alpha| + \epsilon)  \log(|\beta| + \epsilon) } \frac{1}{(\log x)^2}  = \frac{ (\Theta ( \log \log x ))^2  }{ \log( |\alpha| + \epsilon)  \log(|\beta| + \epsilon) } \frac{1}{(\log x)^2} =  \Theta( (  \frac{ \log \log x }{   \log x }  )^2  )
\end{align*}

And 

\begin{align*}
 \log(1+\frac{\epsilon}{|\alpha|}) = \log (1+ \frac{1}{|\alpha|} \cdot \frac{1}{\log x} ) = \Theta (  \frac{ 1 }{   \log x }  ),~~~ \log(1+\frac{\epsilon}{|\beta|}) =  \Theta (  \frac{ 1 }{   \log x }  )  
\end{align*}

It follows 

\begin{align*}
&~~~~\frac{1}{\log( |\alpha| + \epsilon)  \log(|\beta| + \epsilon)} -  \frac{1}{\log( |\alpha| )  \log(|\beta| )}\\\\ &=   -\frac{\log( |\alpha| + \epsilon)  \log(|\beta| + \epsilon) -   \log |\alpha| \log |\beta| }{\log |\alpha| \log |\beta| \log( |\alpha| + \epsilon)  \log(|\beta| + \epsilon)}\\ \\ 
 & = - \frac{\log( |\alpha| + \epsilon)  \log(|\beta| + \epsilon) -\log( |\alpha|)  \log(|\beta| + \epsilon)+\log( |\alpha|)  \log(|\beta| + \epsilon) - \log |\alpha| \log |\beta|}{\log |\alpha| \log |\beta| \log( |\alpha| + \epsilon)  \log(|\beta| + \epsilon)}
 \\ \\ 
 & = -\frac{\log(|\beta|+\epsilon)\log(1+\frac{\epsilon}{|\alpha|}) + \log(1+\frac{\epsilon}{|\beta|})\log |\alpha| }{\log |\alpha| \log |\beta| \log( |\alpha| + \epsilon)  \log(|\beta| + \epsilon)}\\\\&=\Theta (  \frac{ 1 }{   \log x }  ) 
\end{align*}

Combine these terms, we obtain
\begin{align*}
    \frac{S(x)}{  (\log x)^2} &\geqslant \frac{1}{\log |\alpha| \cdot  \log|\beta| } + \Theta (  \frac{ 1 }{   \log x }  )  + \Theta (  \frac{ \log \log x }{   \log x }  )    + \Theta( (  \frac{ \log \log x }{   \log x }  )^2  )\\\\
    & = \frac{1}{\log |\alpha| \cdot  \log|\beta| } + \Theta (  \frac{ \log \log x }{   \log x }  ) 
\end{align*}

\subsection{An upper bound for $S(x)$}
In this subsection, we give estimates for upper bounds for $m$ and $n$ under the condition that $|U_n - G_m |\leqslant x$. Let $x$ be sufficiently large, so that $m$ and $n$ could also  be sufficiently large.

Let's write $a(n) = a_1(n),~~ b(m) = b_1 (m)$ and let $u_n = \sum_{i=2}^{t}a_i(n)\alpha_i^n $ and $g_m = \sum_{j=2}^{s}b_j(m)\beta_j^m$, then

\begin{align*}
   U_n = a(n) \alpha^n + u_n, ~~~~~ G_m = b(m) \beta^m + g_m
\end{align*}

Substitute into $|U_n - G_m |\leqslant x$, then we get

\begin{equation}\label{LongInequality}
   x \geqslant  |a(n) \alpha^n + u_n - (b(m) \beta^m + g_m)|   \geqslant |a(n) \alpha^n - b(m) \beta^m|- |u_n| -|g_m|
\end{equation}
\bigskip

\textbf{Case 1 $m > n$:}  we have \textbf{max}$\{n, m\} = m$.
\bigskip

We first give an upper bound for $m$, then give an upper bound for $n$.

\medskip

Divide $ |b(m)|\cdot|\beta|^m$ in each sides of (\ref{LongInequality}), then

$$ \frac{x}{|b(m)|\cdot|\beta|^m} \geqslant |\frac{a(n)}{b(m)}\alpha^n\beta^{-m} - 1| - \frac{|u_n|}{|b(m)|\cdot |\beta|^m} - \frac{|g_m|}{|b(m)|\cdot|\beta|^m}$$
\bigskip

Set $\Lambda = \frac{a(n)}{b(m)}\alpha^n\beta^{-m} - 1. $ 
\bigskip

Set $t = 3, ~~\gamma_1 = \frac{a(n)}{b(m)},~~ b_1 =  1,~~\gamma_2 = \alpha,~~  b_2 =  n, ~~ \gamma_3 = \beta,~~  b_3 =  -m,~~    $ 
\bigskip

$B = \textbf{max} \{ |n|, |-m|, 1 \} = m.$  $D = [\mathbb{K}: \mathbb{Q} ]$, where $\mathbb{K} = \mathbb{Q}(\alpha, \alpha_2, \dots,\alpha_t, \beta, \beta_2, \dots, \beta_s)$.
\bigskip

$ A_2 \geqslant \textbf{max} \{ D h(\alpha), |\log \alpha|, 0.16 \}, ~~A_3 \geqslant \textbf{max} \{ D h(\beta), |\log \beta|, 0.16 \}$.  So we can choose $A_2$ and $A_3$ to be two positive constants.
\bigskip

By \textbf{Lemma \ref{Loglm}}, there exists a positive constant $C^{'}$, such that when $m \geqslant 2$, we have
$$\textbf{max} \{ D h(\gamma_1), |\log \gamma_1|, 0.16 \} = \textbf{max} \{ D h( \frac{a(n)}{b(m)}), |\log  \frac{a(n)}{b(m)}|, 0.16 \} \leqslant C^{'} \log \textbf{max}\{n, m  \}  = C^{'} \log m,$$

So we can let $A_1 = C^{'} \log m$.
\bigskip

If $\Lambda = 0$, then by \textbf{Lemma \ref{Fraclm}} and \textbf{Lemma \ref{Loglm}}, $m$ will be small~\cite{Pink}.

$$ \Lambda = 0 \Longleftrightarrow \frac{a(n)}{b(m)} = \frac{\beta^m}{\alpha^n} \implies C \log m \geqslant h(\frac{a(n)}{b(m)}) = h(\frac{\beta^m}{\alpha^n}) \geqslant C_0 m \implies m ~~\emph{\emph{is small }} $$
\medskip

Thus we can fix some integer $m_1$ such that $m \geqslant m_1\implies \Lambda \neq 0 $. 
\bigskip

So when $m \geqslant m_1$,  by Matveev’s theorem, we will have 

$$\log |\Lambda| > - C (1+ \log (3m))\log m,$$

where $C$ is a positive constant.
\bigskip

Equivalently, we have

$$|\Lambda| > \emph{\emph{exp}}(- C (1+ \log (3m))\log m) = \frac{1}{m^C} \cdot \frac{1}{m^{C\log 3m}}~,$$
\bigskip

By the conditions that $|\beta| > |\beta_2| $ and $|\beta| > |\alpha_2|$, we can choose $ \beta_2^{\prime}$ and $ \alpha_2 ^{\prime}$ satisfying that  $ |\beta| > |\beta_2^{\prime}|> |\beta_2| $ and $|\beta| > |\alpha_2^{\prime}|> |\alpha_2|.$

\bigskip

Then we can find a positive constant $ \widehat{C}$ such that
 
 $$ |u_n| \leqslant  \widehat{C} \cdot |\alpha_2^{\prime}|^n, ~~~|g_m|\leqslant \widehat{C} \cdot |\beta_2^{\prime}|^m, ~~~~~\forall m, n$$
\medskip

Moreover, since $|\frac{\alpha^{\prime}_2}{\beta}|<1$ and $ |\frac{\beta_2^{\prime}}{\beta}| < 1$, there exists $m_2$ such that when $m \geqslant m_2$, the following will hold.

$$\frac{1}{4m^C} \cdot \frac{1}{m^{C\log 3m}}> \frac{ \widehat{C}}{\delta} \cdot |\frac{\alpha_2^{\prime}}{\beta}|^m ,~~~~~ \frac{1}{4m^C} \cdot \frac{1}{m^{C\log 3m}} > \frac{ \widehat{C}}{\delta} \cdot |\frac{\beta_2^{\prime}}{\beta}|^m$$
\bigskip

And for some positive constant $\delta$ and some integer $m_3$, we have $|b(m)| \geqslant \delta$, for all $ m \geqslant m_3$.

Therefore, when $m \geqslant m_0$,  we can get estimates for lower bounds of $ \frac{x}{|b(m)|\cdot|\beta|^m}$ as following. Here $m_0 = \textbf{max}\{m_1, m_2, m_3\}$.

\begin{align*}
 \frac{x}{|b(m)|\cdot|\beta|^m} &\geqslant |\frac{a(n)}{b(m)}\alpha^n\beta^{-m} - 1| - \frac{|u_n|}{|b(m)|\cdot |\beta|^m} - \frac{|g_m|}{|b(m)|\cdot|\beta|^m}\\
& \geqslant |\Lambda| - \frac{ \widehat{C}}{\delta} \cdot \frac{|\alpha_2^{\prime}|^n}{|\beta|^m} - \frac{ \widehat{C}}{\delta} \cdot \frac{|\beta_2^{\prime}|^m}{|\beta|^m}\\
& > \frac{1}{m^C} \cdot \frac{1}{m^{C\log 3m}} - \frac{ \widehat{C}}{\delta} \cdot |\frac{\alpha_2^{\prime}}{\beta}|^m - \frac{ \widehat{C}}{\delta} \cdot |\frac{\beta_2^{\prime}}{\beta}|^m\\
& =  \frac{1}{2m^C} \cdot \frac{1}{m^{C\log 3m}} +  (\frac{1}{4m^C} \cdot \frac{1}{m^{C\log 3m}}- \frac{ \widehat{C}}{\delta} \cdot |\frac{\alpha_2^{\prime}}{\beta}|^m) + (\frac{1}{4m^C} \cdot \frac{1}{m^{C\log 3m}} - \frac{ \widehat{C}}{\delta} \cdot |\frac{\beta_2^{\prime}}{\beta}|^m)\\
& > \frac{1}{2m^C} \cdot \frac{1}{m^{C\log 3m}}
\end{align*}

Thus 

$$  \frac{x}{|b(m)|\cdot|\beta|^m} > \frac{1}{2m^C} \cdot \frac{1}{m^{C\log 3m}}, ~~ m \geqslant m_0$$

Take logarithm,

\begin{align*}
 \frac{\log x}{\log |\beta|} > m - \frac{C(\log 3m)( \log m)}{\log |\beta|} - \frac{C\log m}{\log |\beta|}+ \frac{ \log \delta -\log 2}{\log |\beta|}    
\end{align*}
\bigskip

The terms  $(\log 3m)( \log m)$ and $\log m$ are controlled by the term $m$, so
$\forall\, \eta \in (0, 1),$ $\exists \,x_1 = x_1 (\eta)$ depending on $\eta$ such that when $x \geqslant x_1$,

$$m \leqslant (1+\eta) \frac{\log x}{\log |\beta|}.$$
\medskip

Next, we will  obtain a bound for $n$.

By $ |G_m| \leqslant K_\epsilon  \cdot (|\beta| + \epsilon)^m $, we have

$$m \leqslant (1+\eta) \frac{\log x}{\log |\beta|} \implies |G_m| \leqslant K_\epsilon \cdot x^{(1+\eta) \frac{\log(|\beta|+\epsilon)}{\log |\beta|}}$$

So  $|U_n|$ is bounded by

$$|U_n| \leqslant |U_n -G_m|+|G_m| \leqslant x + K_\epsilon \cdot x^{(1+\eta) \frac{\log(|\beta|+\epsilon)}{\log |\beta|}} \leqslant 2K_\epsilon \cdot x^{(1+\eta) \frac{\log(|\beta|+\epsilon)}{\log |\beta|}}$$
\bigskip

We can find a positive constant $J_0$ and integer $n_1$ such that 

$$ |U_n| \geqslant J_0 \cdot |\alpha|^n, ~~~\forall n \geqslant n_1$$

Combine these  inequalities 

$$J_0 \cdot |\alpha|^n \leqslant 2K_\epsilon \cdot x^{(1+\eta) \frac{\log(|\beta|+\epsilon)}{\log |\beta|}},   ~~~\forall n \geqslant n_1$$

Take logarithm,

$$\log J_0 + n \log|\alpha| \leqslant \log(2K_\epsilon) + (1+\eta) \frac{\log(|\beta|+\epsilon)}{\log |\beta|} \log x$$

Thus $$n \leqslant  (1+\eta) \frac{\log(|\beta|+\epsilon)}{\log |\beta|}   \frac{\log x}{\log |\alpha|}+  \frac{\log(2K_\epsilon) - \log J_0}{\log |\alpha|} $$
\bigskip

So for the case \textbf{max}$\{n, m\} = m$,
we can get

\begin{align*}
  m &\leqslant  (1+\eta) \frac{\log x}{\log |\beta|}\\
  \\
  n &  \leqslant  (1+\eta) \frac{\log(|\beta|+\epsilon)}{\log |\beta|}   \frac{\log x}{\log |\alpha|}+  \frac{\log(2K_\epsilon) - \log J_0}{\log |\alpha|} 
\end{align*}

\bigskip

\textbf{Case 2 $m \leqslant n$:}  we have  \textbf{max}$\{n, m\} = n$.
\bigskip

Divide $ |a(n)|\cdot|\alpha|^n$ in each sides of (\ref{LongInequality}), then

$$ \frac{x}{|a(n)|\cdot|\alpha|^n} \geqslant |\frac{b(m)}{a(n)}\beta^m\alpha^{-n} - 1| - \frac{|u_n|}{|a(n)|\cdot |\alpha|^n} - \frac{|g_m|}{|a(n)|\cdot|\alpha|^n}$$
\bigskip

Similarly, we can get

\begin{align*}
  n &\leqslant  (1+\eta) \frac{\log x}{\log |\alpha|}\\
  \\
  m &\leqslant  (1+\eta) \frac{\log(|\alpha|+\epsilon)}{\log |\alpha|}   \frac{\log x}{\log |\beta|}+  \frac{\log(2K_\epsilon) - \log J_0}{\log |\beta|}   
\end{align*}

\bigskip

As a result, in both cases \textbf{max}$\{n, m\} = m$ and  \textbf{max}$\{n, m\} = n$, we all have the following inequalities:

\begin{align*}
 m &\leqslant  (1+\eta) \frac{\log(|\alpha|+\epsilon)}{\log |\alpha|}   \frac{\log x}{\log |\beta|}+  \frac{\log(2K_\epsilon) - \log J_0}{\log |\beta|} \\
  \\
  n &\leqslant  (1+\eta) \frac{\log(|\beta|+\epsilon)}{\log |\beta|}   \frac{\log x}{\log |\alpha|}+  \frac{\log(2K_\epsilon) - \log J_0}{\log |\alpha|}    
\end{align*}

\bigskip

Then one can  obtain upper bounds for $S(x)$:

\begin{align*}
 S(x) \leqslant &( (1+\eta) \frac{\log(|\alpha|+\epsilon)}{\log |\alpha|}   \frac{\log x}{\log |\beta|}+  \frac{\log(2K_\epsilon) - \log J_0}{\log |\beta|})
\\\\ &\times ((1+\eta) \frac{\log(|\beta|+\epsilon)}{\log |\beta|}   \frac{\log x}{\log |\alpha|}+  \frac{\log(2K_\epsilon) - \log J_0}{\log |\alpha|} )
\end{align*}
\bigskip

Expand the product, 

\begin{align*}
S(x) \leqslant & (1 + \eta)^2 \frac{\log (|\alpha| + \epsilon) \log (|\beta| + \epsilon)}{\log |\alpha| \log |\beta|} \frac{(\log x)^2}{\log |\alpha| \log |\beta|}\\\\ &+ (1 + \eta) \frac{\log (|\alpha|+\epsilon)}{\log |\alpha|}\frac{\log x}{\log |\beta|} \frac{\log(2K_{\epsilon}) - \log J_0}{\log |\alpha|}\\\\ &+ (1 + \eta) \frac{\log (|\beta|+\epsilon)}{\log |\beta|} \frac{\log x}{\log |\alpha|} \frac{\log(2K_{\epsilon}) - \log J_0}{\log |\beta|}\\\\ 
&+ \frac{(\log(2K_{\epsilon}) - \log J_0)^2}{\log |\alpha| \log |\beta|} \,.
\end{align*}

Again we set $\epsilon = \frac{1}{\log x} $ and set 

\begin{align*}
\mathbb{I}:~ = (1 + \eta)^2 \frac{\log (|\alpha| + \epsilon) \log (|\beta| + \epsilon)}{\log (|\alpha|) \log(|\beta|) } \frac{1}{\log (|\alpha|) \log(|\beta|) }   
\end{align*}

then

\begin{align*}
 \frac{ \mathbb{I}}{(1 + \eta)^2} & =  \frac{\log (|\alpha| + \epsilon) \log (|\beta| + \epsilon)- \log (|\alpha|) \log(|\beta|)}{(\log |\alpha| \log|\beta| )^2} + \frac{1}{\log |\alpha| \log|\beta| }\\   &= \Theta (\frac{1}{\log x}) + \frac{1}{\log |\alpha| \log|\beta| } 
\end{align*}

Set 

\begin{align*}
\mathbb{II}:~ = (1 + \eta) \frac{\log (|\alpha|+\epsilon)}{\log |\alpha|\log |\beta|} \frac{\log(2K_{\epsilon}) - \log J_0}{\log |\alpha|} \frac{1}{\log x}
\end{align*}

then by $\log(2 K_{\epsilon}) =\Theta ( \log \log x )$, we will have

\begin{align*}
\mathbb{II} = \Theta (\frac{ \log \log x}{\log x} )
\end{align*}

Similarly, one can get

\begin{align*}
\mathbb{II}^{\prime}:~ = (1 + \eta) \frac{\log (|\beta|+\epsilon)}{\log |\beta|\log |\alpha|}  \frac{\log(2K_{\epsilon}) - \log J_0}{\log |\beta|}\frac{1}{\log x} = \Theta (\frac{ \log \log x}{\log x} )
\end{align*}

And 

\begin{align*}
 \mathbb{III}:~ =   \frac{(\log(2K_{\epsilon}) - \log J_0)^2}{\log |\alpha| \log |\beta|} \frac{1}{(\log x)^2} = \Theta( (  \frac{ \log \log x }{   \log x }  )^2  )
\end{align*}

As a result,
\begin{align*}
  \frac{S(x)}{  (\log x)^2} &\leqslant \mathbb{I} +\mathbb{II} + \mathbb{II}^{\prime} + \mathbb{III}\\
  & = (\Theta (\frac{1}{\log x}) + \frac{(1 + \eta)^2}{\log |\alpha| \log|\beta| }) + \Theta (\frac{ \log \log x}{\log x} ) +\Theta (\frac{ \log \log x}{\log x} )+ \Theta( (  \frac{ \log \log x }{   \log x }  )^2  ) \\\\
  & = (1 + \eta)^2 \frac{1}{\log |\alpha| \log|\beta| } + \Theta (\frac{ \log \log x}{\log x} )
\end{align*}

So
\begin{align*}
  \frac{S(x)}{  (\log x)^2} &\leqslant  (1 + \eta)^2 \frac{1}{\log |\alpha| \cdot \log|\beta| } + \Theta (\frac{ \log \log x}{\log x} ), ~~~~\forall x \geqslant x_1(\eta)
\end{align*}

\end{proof}
\bigskip

\section{Further Conjectures and Problems}
The key in our proof is the tool for linear logarithmic forms  for algebraic numbers $\alpha$ and $\beta$. However, we conjecture that there should be similar results even when  $\alpha$ or $\beta$ are transcendental numbers.
\bigskip




\begin{conjecture}
Let $\alpha, \beta \in \mathbb{C},~ |\alpha|>1, ~ |\beta| > 1$.  Then $\alpha$ and $\beta$ are multiplicatively independent  if and only if

$$ \#\{(n, m): ~  | \alpha^n - \beta^m | \leqslant x,~ (n, m) \in \N \times \N \}  \sim \frac{(\log x)^2}{\log |\alpha|\cdot \log  |\beta| },~~~~~ x \to \infty$$
\end{conjecture}
\bigskip

We also have the following conjecture which can imply that $\log(\pi)$ is irrational.
\begin{conjecture}
$$ \#\{(n, m): ~  | \pi^n - e^m | \leqslant x, ~(n, m) \in \N \times \N \}  \sim \frac{(\log x)^2}{\log \pi},~~~~~ x \to \infty$$
\end{conjecture}

We also put some problems for more than two powers. These questions may also need linear logarithmic forms for transcendental numbers, which is beyond our current methods. 
\bigskip

\textbf{Problem A:} Is the following true?

$$ \#\{(n, m): ~  | \pi^n + (\sqrt 5)^n - 7^m - e^m| \leqslant x, ~(n, m) \in \N \times \N \}  \sim \frac{(\log x)^2}{(\log \pi) (\log 7)},~~~~~ x \to \infty$$

\bigskip

\textbf{Problem B:} Could one find positive constants $\alpha, \beta, \xi, M$, all larger than 1, such that   any two of $\alpha, \beta, \xi$ are  multiplicatively independent and  the following can be true?

$$ \#\{(n, m, k): ~  | \alpha^n + \beta^m - \xi^k| \leqslant M, ~(n, m, k) \in \N \times \N \times \N\} =  \infty$$
\bigskip


\textbf{Acknowledgements.}\;The author acknowledges the support of the Austrian Science Fund (FWF): W1230. The author also thanks useful suggestions from professor Robert Tichy for finding the remainder term.

\noindent
\textsc{Institute of Analysis and Number Theory, Graz University of Technology, Kopernikusgasse 24/II,
8010 Graz, Austria}
\medskip

\noindent
\emph{E-mail address:} yang@tugraz.at

\end{document}